\documentclass[12pt]{article}
\input epsf.tex

\usepackage{amsmath}
\usepackage{amsthm}
\usepackage{amsfonts}
\usepackage{amssymb}
\usepackage{graphicx}
\usepackage{latexsym}

\usepackage{amsmath,amsthm,amsfonts,amssymb}

\usepackage{epsfig}

\usepackage{amssymb}
\usepackage[all]{xy}
\xyoption{poly}
\usepackage{fancyhdr}
\usepackage{wrapfig}

\theoremstyle{plain}
\newtheorem{thm}{Theorem}[section]
\newtheorem{prop}[thm]{Proposition}
\newtheorem{lem}[thm]{Lemma}

\theoremstyle{definition}
\newtheorem{defn}{Definition}
\theoremstyle{remark}

\advance \topmargin by -\headheight
\advance \topmargin by -\headsep
\evensidemargin \oddsidemargin
\def\sF{{\mathcal F}}

\def\N{{\mathbb N}}

\def\sP{{\mathcal P}}

\def\R{{\mathbb R}}

\def\chix{{\raise.5ex\hbox{$\chi$}}}

\def\Z{{\mathbb Z}}

\begin{document}
\title{A measure-conjugacy invariant for free group actions}
\author{Lewis Bowen\footnote{email:lpbowen@math.hawaii.edu} \\ University of Hawaii}
\begin{abstract}
This paper introduces a new measure-conjugacy invariant for actions of free groups. Using this invariant, it is shown that two Bernoulli shifts over a finitely generated free group are measurably conjugate if and only if their base measures have the same entropy. This answers a question of Ornstein and Weiss.
\end{abstract}
\maketitle

\noindent

\section{Introduction}

This paper is motivated by an old and central problem in measurable dynamics: given two dynamical systems, determine whether they are measurably-conjugate, i.e., isomorphic. Let us set some notation.

A {\bf dynamical system} (or system for short) is a triple $(G,X,\mu)$ where $(X,\mu)$ is a probability space and $G$ is a group acting by measure-preserving transformations on $(X,\mu)$. We will also call this a {\bf dynamical system over $G$}, a {\bf $G$-system} or an {\bf action of $G$}. In this paper, $G$ will always be a discrete countable group. Two systems $(G,X,\mu)$ and $(G,Y,\nu)$ are {\bf isomorphic} (i.e., {\bf measurably conjugate}) if and only if there exist conull sets $X' \subset X, Y' \subset Y$ and a bijective measurable map $\phi:X'\to Y'$ such that $\phi^{-1}$ is measurable, $\phi_*\mu=\nu$ and $\phi(gx)=g\phi(x) \forall g\in G, x \in X'$. 

A special class of dynamical systems called Bernoulli systems or Bernoulli shifts has played a significant role in the development of the theory as a whole; it was the problem of trying to classify them that motivated Kolmogorov to introduce the mean entropy of a dynamical system over $\Z$ [Ko58, Ko59]. That is, Kolmogorov defined for every system $(\Z,X,\mu)$ a number $h(\Z,X,\mu)$ called the {\bf mean entropy} of $(\Z,X,\mu)$ that quantifies, in some sense, how ``random'' the system is. See [Ka07] for a historical survey.

Bernoulli shifts also play an important role in this paper, so let us define them. Let $(K,\kappa)$ be a standard Borel probability space. For a discrete countable group $G$, let $K^G = \prod_{g \in G} K$ be the set of all functions $x: G \to K$ with the product Borel structure and let $\kappa^G$ be the product measure on $K^G$. $G$ acts on $K^G$ by $(gx)(f)=x(g^{-1}f)$ for $x \in K^G$ and $g,f \in G$. This action is measure-preserving. The system $(G,K^G,\kappa^G)$ is the {\bf Bernoulli shift over $G$ with base $(K,\kappa)$}. It is nontrivial if $\kappa$ is not supported on a single point.

Before Kolmogorov's seminal work [Ko58, Ko59], it was unknown whether all nontrivial Bernoulli shifts over $\Z$ were measurably conjugate to each other. He proved that $h(\Z,K^\Z,\kappa^\Z)=H(\kappa)$ where $H(\kappa)$, the {\bf entropy of $\kappa$} is defined as follows. If there exists a finite or countably infinite set $K' \subset K$ such that $\kappa(K')=1$ then
$$H(\kappa)= - \sum_{k\in K'} \mu(\{k\}) \log( \mu(\{k\}) )$$
where we follow the convention $0\log(0)=0$. Otherwise, $H(\kappa)=+\infty$. Thus two Bernoulli shifts over $\Z$ with different base measure entropies cannot be measurably conjugate. 

The converse was proven by D. Ornstein in the groundbreaking papers [Or70a, Or70b]. That is, he proved that if two Bernoulli shifts $(\Z, K^\Z, \kappa^\Z), (\Z,L^\Z,\lambda^\Z)$ are such that $H(\kappa)=H(\lambda)$ then they are isomorphic.

 In [Ki75], Kieffer proved a generalization of the Shannon-McMillan theorem to actions of a countable amenable group $G$. In particular, he extended the definition of mean entropy from $\Z$-systems to $G$-systems. This leads to the generalization of Kolmogorov's theorem to amenable groups.

In the landmark paper [OW87], Ornstein and Weiss extended most of the classical entropy theory from $\Z$-systems to $G$-systems where $G$ is any countable amenable group. In particular, they proved that if two Bernoulli shifts $(G, K^G, \kappa^G)$, $(G,L^G,\lambda^G)$ over a countably infinite amenable group $G$ are such that $H(\kappa)=H(\lambda)$ then they are isomorphic. Thus Bernoulli shifts over $G$ are completely classified by base measure entropy.

Now let us say that a group $G$ is {\bf Ornstein} if $H(\kappa)=H(\lambda)$ implies $(G, K^G, \kappa^G)$ is isomorphic to $(G,L^G,\lambda^G)$ whenever $(K,\kappa)$ and $(L,\lambda)$ are standard Borel probability spaces. By the above, all countably infinite amenable groups are Ornstein. Stepin proved that any countable group that contains an Ornstein subgroup is itself Ornstein [St75]. It is unknown whether every countably infinite group is Ornstein.

In [OW87], Ornstein and Weiss asked whether all Bernoulli shifts over a nonamenable group are isomorphic. The next result shows that the answer is `no':
\begin{thm}\label{thm:free}
Let $G=\langle s_1,\ldots,s_r \rangle$ be the free group of rank $r$. If $(K_1,\kappa_1), (K_2,\kappa_2)$ are standard probability spaces with $H(\kappa_1)+H(\kappa_2)<\infty$ then $(G,K^G_1,\kappa_1^G)$ is measurably conjugate to $(G,K^G_2,\kappa_2^G)$ if and only if $H(\kappa_1)=H(\kappa_2)$.
\end{thm}

The reason Ornstein and Weiss thought the answer might be `yes' is due to a curious example presented in [OW87]. It pertains to a well-known fundamental property of entropy: it is nonincreasing under factor maps. Let $(G,X,\mu)$ and $(G,Y,\nu)$ be two systems. A map $\phi:X \to Y$ is a factor if $\phi_*\mu=\nu$ and $\phi(gx)=g\phi(x)$ for a.e. $x\in X$ and every $g\in G$. If $G$ is amenable then the mean entropy of a factor is less than or equal to the mean entropy of the source. This is essentially due to Sinai. So if $K_n=\{1,\dots,n\}$ and $\kappa_n$ is the uniform probability measure on $K_n$ then $(G,K_2^G,\kappa_2^G)$, which has entropy $\log(2)$, cannot factor onto $(G,K_4^G,\kappa_4^G)$, which has entropy $\log(4)$.

The argument above fails if $G$ is nonamenable. Indeed, let $G=\langle a,b \rangle$ be a rank 2 free group. Identify $K_2$ with the group $\Z/2\Z$ and $K_4$ with $\Z/2\Z \times \Z/2\Z$. Then 
$$\phi(x)(g) := \big(x(g)+x(ga), x(g) + x(gb) \big) ~\forall x \in K_2^G, g\in G$$
is a factor map from $(G,K_2^G,\kappa_2^G)$ onto $(G,K_4^G,\kappa_4^G)$. This is Ornstein-Weiss' example. It is the main ingredient in the proof of the next theorem, which will appear in a separate paper.
\begin{thm}
Let $G$ be any countable group that contains a nonabelian free subgroup. Then every nontrivial Bernoulli shift over $G$ factors onto every Bernoulli shift over $G$. 
\end{thm}

To prove theorem \ref{thm:free}, the following invariant is introduced. Let $(X,\mu)$ be any probability space on which $G=\langle s_1,\ldots,s_r\rangle$, the rank $r$ free group, acts by measure-preserving transformations. Let $\alpha=\{A_1,\ldots,A_n\}$ be a partition of $X$ into finitely many measurable sets. Let $B(e,n) \subset G$ denote the ball of radius $n$ centered at the identity element with respect to the word metric induced by $S=\{s_1^{\pm 1}, \ldots, s_r^{\pm 1}\}$. The join of two partitions $\alpha, \beta$ of $X$ is defined by $\alpha \vee \beta = \{A \cap B ~|~ A \in \alpha, B \in \beta\}$. Let 
\begin{eqnarray*}
H(\alpha) &:=& -\sum_{A \in \alpha} \mu(A)\log(\mu(A)),\\ 
F(\alpha) &:=& (1-2r)H(\alpha)+ \sum_{i =1}^r H(\alpha \vee s_i\alpha),\\
\alpha^n &:=& \bigvee_{g\in B(e,n)} g\alpha,\\
f(\alpha) &:=& \inf_{n} F(\alpha^n).
\end{eqnarray*}
 A partition $\alpha$ is {\bf generating} if the smallest $G$-invariant $\sigma$-algebra containing $\alpha$ is the $\sigma$-algebra of all measurable sets (up to sets of measure zero). The main theorem of this paper is:
\begin{thm}\label{thm:main}
Let $G=\langle s_1,\ldots,s_r\rangle$. Let $(G,X,\mu)$ be a system. If $\alpha$ and $\beta$ are finite measurable generating partitions of $X$ then $f(\alpha)=f(\beta)$. Hence this number, denoted $f(G,X,\mu)$, is a measure-conjugacy invariant.
\end{thm}

Theorem \ref{thm:bernoulli} below implies that if $|K|<\infty$ then $f(G,K^G,\kappa^G)=H(\kappa)$. This and Stepin's theorem proves theorem \ref{thm:free}. A simple exercise reveals that if $r=1$, then $f(G,X,\mu)=h(G,X,\mu)$ is Kolmogorov's entropy.

Here is a brief outline of the paper. In the next section, standard entropy-theory definitions are presented. In \S \ref{sec:topological}, an equivalence relation, called combinatorial equivalence, is introduced on the space of finite partitions of $X$, where $(X,\mu)$ is a standard probability space on which a countable group $G$ acts. We prove that the combinatorial equivalence class of a finite generating partition is dense in the space of all generating partitions. In \S \ref{sec:splittings}, we introduce an operation on partitions called splitting and show that any two combinatorially equivalent partitions have a common splitting. This culminates in a condition sufficient for a function from the space of partitions to $\R$ to induce a measure-conjugacy invariant. In \S \ref{sec:f}, this condition is shown to hold for the function $F$ defined above. This proves theorem \ref{thm:main}. Then we prove theorem \ref{thm:bernoulli} (that $f(G,K^G,\kappa^G)=H(\kappa)$ if $|K|<\infty$) and conclude theorem \ref{thm:free}.

{\bf Acknowledgments}: I am grateful to Russell Lyons for suggesting this problem, for encouragements and for many helpful conversations. I would like to thank Dan Rudolph for pointing out errors in a preliminary version. I would like to thank Benjy Weiss for many comments that helped me to greatly improve the exposition of this paper and simplify several arguments.

\section{Some Standard Definitions}\label{sec:standard}


For the rest of this section, fix a standard probability space $(X,\mu)$. 
\begin{defn}
A {\bf partition} $\alpha=\{A_1,\ldots,A_n\}$ is a pairwise disjoint collection of measurable subsets $A_i$ of $X$ such that $\cup_{i=1}^n A_i = X$. The sets $A_i$ are called the {\bf partition elements} of $\alpha$. Alternatively, they are called the {\bf atoms} of $\alpha$. Unless stated otherwise, all partitions in this paper are finite (i.e., $n<\infty$).

If $\alpha$ and $\beta$ are partitions of $X$ then we write $\alpha=\beta$ a.e. if for all $A\in \alpha$ there exists $B\in\beta$ with $\mu(A \Delta B)=0$. Let $\sP$ denote the set of all a.e.-equivalence classes of finite partitions of $X$. By a standard abuse of notation, we will refer to elements of $\sP$ as partitions.


\end{defn}

\begin{defn}
If $\alpha, \beta \in \sP$ then the {\bf join}  of $\alpha$ and $\beta$ is the partition $\alpha \vee \beta=\{A \cap B\,|\,A \in \alpha, ~B\in \beta\} $.
\end{defn}



\begin{defn}
Let $\sF$ be a $\sigma$-algebra contained in the $\sigma$-algebra of all measurable subsets of $X$. Given a partition $\alpha$, define the {\bf conditional information function} $I(\alpha|\sF):X \to \R$ by
$$I(\alpha|\sF)(x) = -\log\big(\mu(A_x|\sF)(x)\big)$$
where $A_x$ is the atom of $\alpha$ containing $x$. Here $\mu(A_x|\sF):X \to \R$ is the conditional expectation of $\chi_{A_x}$, the characteristic function of $A_x$, with respect to the $\sigma$-algebra $\sF$.
 
The {\bf conditional entropy of $\alpha$ with respect to $\sF$} is defined by
$$H(\alpha|\sF) = \int_X I(\alpha | \sF)(x) \, d\mu(x).$$

If $\beta$ is a partition then, by abuse of notation, we can identify $\beta$ with the $\sigma$-algebra equal to the set of all unions of partition elements of $\beta$. Through this identification, $I(\alpha|\beta)$ and $H(\alpha|\beta)$ are well-defined. Let $I(\alpha)=I(\alpha|\{\emptyset,X\})$ and $H(\alpha)=H(\alpha|\{\emptyset,X\})$.

\end{defn}

\begin{lem}\label{lem:relative}
For any two partitions $\alpha, \beta$ and for any two $\sigma$-algebras $\sF_1, \sF_2$ with $\sF_1 \subset \sF_2$,
$$H(\alpha \vee \beta) = H(\alpha) + H(\beta|\alpha),$$
$$H(\alpha  | \sF_2) \le H(\alpha  | \sF_1).$$
\end{lem}
\begin{proof}
This is well-known. For example, see [Gl03, Proposition 14.16, page 255]. 
\end{proof}




\begin{defn}[Rokhlin distance]
Define $d:\sP \times \sP \to \R$ by
$$d(\alpha,\beta) = H(\alpha|\beta) + H(\beta|\alpha) = 2H(\alpha \vee \beta) - H(\alpha) - H(\beta).$$
By [Pa69, theorem 5.22, page 62] this defines a distance function on $\sP$. If $G$ is a group acting by measure-preserving transformations on $(X,\mu)$ then the action of $G$ on $\sP$ is isometric. I.e., if $g \in G$, $\alpha, \beta \in \sP$ then $d(g\alpha, g\beta) = d(\alpha,\beta)$. From now on, we consider $\sP$ with the topology induced by $d(\cdot,\cdot)$.
\end{defn}



\begin{defn}
Let $G$ be a group acting on $(X,\mu)$. Let $\alpha$ be a partition of $X$. Let $\Sigma_\alpha$ be the smallest $G$-invariant $\sigma$-algebra containing $\alpha$. Then $\alpha$ is {\bf generating} if for any measurable set $A \subset X$ there exists a set $A' \in \Sigma_\alpha$ such that $\mu(A \Delta A')=0$. Let $\sP_{gen} \subset \sP$ denote the set of all generating partitions.
\end{defn}

\section{Combinatorially Equivalent Partitions}\label{sec:topological}

For this section, fix a countable group $G$ and an action of $G$ on a standard probability space $(X,\mu)$ by measure-preserving transformations.

\begin{defn}
Given $\alpha \in \sP$ and $F \subset G$ finite, let $\alpha^F = \bigvee_{f \in F} f\alpha$. 
\end{defn}

\begin{defn}
If $\alpha, \beta \in \sP$ are such that for all $A \in \alpha$ there exists $B \in \beta$ with $\mu(A-B)=0$ then we say that $\alpha$ {\bf refines $\beta$} and denote it by $\alpha \ge \beta$. Equivalently, $\beta$ is a {\bf coarsening} of $\alpha$.
\end{defn}

\begin{defn}\label{def:top}
Let $\alpha, \beta \in \sP$. We say that $\alpha$ is {\bf combinatorially equivalent} to $\beta$ if there exist finite sets $L,M \subset G$ such that $\alpha \le \beta^L$ and $\beta \le \alpha^M$. Let $\sP_{eq}(\alpha) \subset \sP$ denote the set of partitions combinatorially equivalent to $\alpha$
\end{defn}


 
 The goal of this section is to prove theorem \ref{thm:dense} below: if $\alpha$ is a generating partition then $\sP_{eq}(\alpha)$ is dense in the subspace of all generating partitions.




\begin{lem}\label{lem:approx1}
Let $\alpha$ be a generating partition and $\beta=\{B_1,\ldots,B_m\} \in \sP$. Let $\epsilon>0$. Then there exists a partition $\beta'=\{B'_1,\ldots,B'_m\}$ and a finite set $L \subset G$ such that $\alpha^L \ge \beta'$ and for all $i=1\ldots m$, $\mu(B_i \Delta B'_i) \le \epsilon$.  
\end{lem}

\begin{proof}
Since $\alpha$ is generating, there exists a finite set $L \subset G$ such that for every $i\in\{1,\ldots,m\}$, there is a set $B''_i$, equal to a finite union of atoms of $\alpha^L$, such that $\mu(B_i \Delta B''_i) < \frac{\epsilon}{m}$. For $i=1\ldots m-1$, let 
$$B'_i := B''_i - \bigcup_{j\ne i} B''_j.$$ 
$$B'_m := X - \bigcup_{i < m} B'_i = B''_m \cup \bigcup_{i\ne j} B''_i \cap B''_j.$$
Observe that for all $i=1\ldots m$,
$$B_i - \bigcup_j B''_j \Delta B_j \subset B'_i \subset  B_i \cup \bigcup_j B''_j \Delta B_j.$$
Thus
$$\mu(B'_i \Delta B_i) \le m\Big(\frac{\epsilon}{m}\Big) = \epsilon.$$
By construction, $\beta'=\{B'_1,\ldots,B'_m\} \le \alpha^L$.





\end{proof}




\begin{lem}
Let $\alpha=\{A_1,\ldots,A_n\} \in \sP$ and $\beta \in \sP_{gen}$. Let $\epsilon>0$. Then there exists a finite set $M\subset G$ such that for all finite $L \subset G$ with $M \subset L$, the partition elements $\{B^L_1,\ldots,B^L_{m_L}\}$ of $\beta^L$ can be ordered so that there exists an $r\in \{1,\ldots,m_L\}$ and a function $f:\{1,2,\ldots r\} \to \{1,2,\ldots,n\}$ so that for all $i \in \{1,\ldots,r\}$,
$$\frac{\mu(B^L_i \cap A_{f(i)})}{\mu(B^L_i)} \ge 1-\epsilon$$
and
$$\mu\Big( \bigcup_{i > r} B^L_i\Big) < \epsilon.$$
\end{lem}

\begin{proof}
Let $\delta>0$ be such that $\delta < \Big(\frac{\epsilon}{n}\Big)^2$. By the previous lemma, there exists a partition $\alpha' =\{A'_1,\ldots,A'_n\} \in \sP$ and a finite set $M \subset G$ such that $\alpha' \le \beta^M$ and $\mu(A'_i\Delta A_i)<\delta$ for all $i$. Let $L$ be any finite subset of $G$ with $M \subset L$.


Let $\beta^L=\{B^L_1,\ldots,B^L_{m_L}\}$ and let $f:\{1,\ldots,{m_L}\} \to \{1,\ldots,n\}$ be the function $f(i)=j$ if $\mu(B^L_i -A'_j)=0$. This is well-defined since $\beta^L$ refines $\alpha'$.

After reordering the partition elements of $\beta^L=\{B^L_1,\ldots,B^L_{m_L}\}$ if necessary, we may assume that there is an $r \in \{0,\ldots,{m_L}\}$ such that, if $r>0$ then for all $ i \le r$,
$$\frac{\mu(B^L_i \cap A_{f(i)})}{\mu(B^L_i)} \ge 1-\sqrt{\delta},$$
and if $i>r$ then 
$$\frac{\mu(B^L_i \cap A_{f(i)})}{\mu(B^L_i)} < 1-\sqrt{\delta}.$$
So if $i>r$ then
$$\mu(B^L_i \cap A_{f(i)})< (1-\sqrt{\delta})\mu(B^L_i).$$
So
\begin{eqnarray*}
\mu(B^L_i) &=& \mu(B^L_i - A_{f(i)}) + \mu(B^L_i \cap A_{f(i)})\\ 
        &<& \mu(B^L_i - A_{f(i)})+ (1-\sqrt{\delta})\mu(B^L_i).
\end{eqnarray*}
Solve for $\mu(B^L_i)$ to obtain
\begin{eqnarray*}
\mu(B^L_i) &<& \frac{1}{\sqrt{\delta}} \mu(B^L_i - A_{f(i)}).
\end{eqnarray*}
Since the atoms $B^L_i$ are pairwise disjoint, it follows that
\begin{eqnarray*}
\mu\Big(\bigcup_{i>r} B^L_i\Big) &<& \frac{1}{\sqrt{\delta}} \mu\Big(\bigcup_{i>r} B^L_i - A_{f(i)}\Big).
\end{eqnarray*}
Since $\mu(B^L_i - A'_{f(i)})=0$, it must be that $B^L_i - A_{f(i)} \subset A'_{f(i)} - A_{f(i)}$, up to a set of measure zero. So,
\begin{eqnarray*}
\mu\Big( \bigcup_{i>r} B^L_i \Big) &\le &  \frac{1}{\sqrt{\delta}} \mu\Big(\bigcup_{i>r} A'_{f(i)} - A_{f(i)}\Big)\\
&\le& n\sqrt{\delta}< \epsilon.
\end{eqnarray*}

\end{proof}

\begin{thm}\label{thm:dense}
If $\alpha$ is a generating partition then 
$$\sP_{gen} \subset \overline{ \sP_{eq}(\alpha)}.$$
I.e., the subspace of partitions combinatorially equivalent to $\alpha$ is dense in the space of all generating partitions.
\end{thm}

\begin{proof}
Let $\alpha=\{A_1,\ldots,A_n\}$ and $\beta =\{B_1,\ldots,B_m\} \in \sP_{gen}$. Without loss of generality, we assume that $\mu(A_i)>0$ for all $i=1\ldots n$. Let $\epsilon>0$. By the previous lemma, there exists a finite set $L \subset G$ such that the atoms of $\beta^L=\{B^L_1,\ldots,B^L_{m_L}\}$ can be ordered so that there exists an $r\in \{1,\ldots,m_L\}$ and a function $f:\{1,2,\ldots r\} \to \{1,2,\ldots,n\}$ so that for all $i \in \{1,\ldots,r\}$,
$$\frac{\mu(B^L_i \cap A_{f(i)})}{\mu(B^L_i)} \ge 1-\epsilon$$
and 
\begin{eqnarray}\label{eqn:3}
\mu\Big( \bigcup_{i > r} B^L_i\Big) < \epsilon.
\end{eqnarray}
By choosing $\epsilon$ small enough (if necessary) we may assume that $f$ is onto (for example, by choosing $\epsilon$ to be smaller than $\frac{1}{2}\mu(A_j)$ over all $j=1\ldots n$).

By definition of $\beta^L$, $m_L \le m^{|L|}$. If necessary, we may assume that $m_L=m^{|L|}$ after modifying $\beta^L$ by adding to it several copies of the empty set. That is, for some $i$, it may occur that $B^L_i =\emptyset$.

Let $\delta>0$ be such that $\delta<\epsilon$. By lemma \ref{lem:approx1} there exists a partition $\gamma=\{C_1,\ldots,C_m\}$ and a finite set $M\subset G$ such that $\gamma \le \alpha^M$ and $\mu(C_i \Delta B_i) < \delta$ for all $i$. By choosing $\delta$ small enough we may assume the following. Let $\gamma^L=\{C^L_1,\ldots,C^L_{m_L}\}$. Then, after reordering the atoms of $\gamma^L$ if necessary,
\begin{eqnarray}\label{eqn:1}
\mu\Big(\bigcup_{j=1}^{m_L}  C^L_j - B_j^L \Big) \le \epsilon.
\end{eqnarray}

Let 
\begin{eqnarray*}
C'_i &=& \{x \in C_i \, |\, \textrm{ if } x \in C^L_j \textrm{ for some } j \textrm{ then } x \in A_{f(j)}\}\\
&=& \bigcup_{j=1}^{m_L}\, C_i \cap C^L_j \cap A_{f(j)}.
\end{eqnarray*}

Let $C_{i,j} = C_i \cap A_j - C'_i$. Let 
$$\gamma_1=\{C'_i\,|\, i=1\ldots m\} \cup \{C_{i,j}\,|\, 1 \le i,j\le m\}.$$

Note that $\gamma_1 \le (\alpha^M)^L = \alpha^{LM}$ where $LM=\{lm~|~l \in L, m\in M\}$. We claim that $\gamma_1$ is combinatorially equivalent to $\alpha$. Let $\Sigma_1$ be the smallest $G$-invariant collection of subsets of $X$ that is closed under finite intersections and unions and contains the atoms of $\gamma_1$. It suffices to show that every atom of $\alpha$ is in $\Sigma_1$. Observe that, for each $i$,  $C_i=C'_i \cup \bigcup_{j=1}^m C_{i,j}$. Hence, $C_i \in \Sigma_1$. Therefore the atoms of $\gamma^L$ are also in $\Sigma_1$. Since $f$ is onto, the definition of $C'_i$ implies
$$C'_i\cap A_{p} = \cup\{ C'_i \cap C^L_j ~|~ f(j)=p\}.$$
So $C'_i \cap A_{p}$ is in $\Sigma_1$ for all $i,p$. Now $C_i \cap A_p = C_{i,p} \cup (C'_i \cap A_p)$. So $C_i\cap A_p \in \Sigma_1$ for all $i,p$. Since
 $$A_{p} = \bigcup_{i=1}^m C_i \cap A_{p},$$
 $A_p \in \Sigma_1$. Since $p$ is arbitrary, this proves the claim. Thus $\gamma_1 \in \sP_{eq}(\alpha)$.

We claim that $\mu(C'_i \Delta C_i) \le 3\epsilon$ for all $i$.
By definition,
$$C'_i\Delta C_i = C_i - C'_i \subset \bigcup_{j=1}^{m_L} C^L_j - A_{f(j)}.$$
For each $j$,
$$C^L_j - A_{f(j)} \subset (C^L_j - B^L_j) \cup (B^L_j - A_{f(j)}).$$
Thus,
\begin{eqnarray}\label{eqn:0}
C'_i \Delta C_i \subset \bigcup_{j=1}^{m_L}(C^L_j - B^L_j) \cup\bigcup_{j=1}^{r}(B^L_j - A_{f(j)})\cup\bigcup_{j>r}(B^L_j - A_{f(j)}).
\end{eqnarray}
If $j\le r$, then by definition of $r$,
\begin{eqnarray*}
\frac{\mu(B^L_j \cap A_{f(j)})}{\mu(B^L_j)} \ge 1-\epsilon.
\end{eqnarray*}
This implies
\begin{eqnarray*}
\mu(B^L_j - A^L_{f(j)}) \le \epsilon \mu(B^L_j).
\end{eqnarray*}
Thus 
\begin{eqnarray}\label{eqn:2}
\mu\Big(\bigcup_{j=1}^{r} B^L_j - A^L_{f(j)} \Big) \le \sum_j \epsilon \mu(B^L_j) \le \epsilon.
\end{eqnarray}
Equations \ref{eqn:0}, \ref{eqn:1},  \ref{eqn:2} and \ref{eqn:3} imply the claim.

Since $\delta<\epsilon$ and  $\mu(C_i \Delta B_i) < \delta$ for all $i$, the above claim implies that $\mu(C'_i\Delta B_i) \le 4\epsilon$ for all $i$. Thus we have shown that for every $\epsilon>0$, there exists a partition $\gamma_1=\{C'_1,\ldots,C'_m,\ldots\}$, combinatorially equivalent to $\alpha$, containing at most $m+m^2$ partition elements and such that $\mu(C'_i \Delta B_i) < 4\epsilon$ for $i=1\ldots m$. This implies that $\beta$ is in the closure of $\sP_{eq}(\alpha)$. Since $\beta$ is arbitrary this implies the theorem.
\end{proof}

\section{Splittings}\label{sec:splittings}







In this section, $G$ can be any finitely generated group with finite symmetric generating set $S$. Let $(X,\mu)$ be a standard probability space on which $G$ acts by measure-preserving transformations.


\begin{defn}\label{defn:splitting}
Let $\alpha$ be a partition. A {\bf simple splitting} of $\alpha$ is a partition $\sigma$ of the form $\sigma=\alpha \vee s\beta$ where $s\in S$ and $\beta$ is a coarsening of $\alpha$. 

A {\bf splitting} of $\alpha$ is any partition $\sigma$ that can be obtained from $\alpha$ by a sequence of simple splittings. In other words, there exist partitions $\alpha_0, \alpha_1,\ldots,\alpha_m$ such that $\alpha_0=\alpha$, $\alpha_m = \sigma$ and $\alpha_{i+1}$ is a simple splitting of $\alpha_i$ for all $1 \le i < m$.

If $\sigma$ is a splitting of $\alpha$ then $\alpha$ and $\sigma$ are combinatorially equivalent. The splitting concept originated from Williams' work [Wi73] in symbolic dynamics.
\end{defn}

\begin{defn}
The {\bf Cayley graph} $\Gamma$ of $(G,S)$ is defined as follows. The vertex set of $\Gamma$ is $G$. For every $s \in S$ and every $g \in G$ there is a directed edge from $g$ to $gs$ labeled $s$. There are no other edges.

The {\bf induced subgraph} of a subset $F \subset G$ is the largest subgraph of $\Gamma$ with vertex set $F$. A subset $F \subset G$ is {\bf connected} if its induced subgraph in $\Gamma$ is connected.
\end{defn}

\begin{lem}\label{lem:splittings}
If $\alpha, \beta \in \sP$, $\alpha$ refines $\beta$ and $F \subset G$ is finite, connected and contains the identity element $e$ then
$$\alpha \vee \bigvee_{f \in F^{-1}} f\beta$$
is a splitting of $\alpha$.
\end{lem}

\begin{proof}
We prove this by induction on $|F|$. If $|F|=1$ then $F=\{e\}$ and the statement is trivial. Let $f_0 \in F-\{e\}$ be such that $F_1=F-\{f_0\}$ is connected. To see that such an $f_0$ exists, choose a spanning tree for the induced subgraph of $F$. Let $f_0$ be any leaf of this tree that is not equal to $e$. 

By induction, $\alpha_1 := \alpha \vee \bigvee_{f \in F^{-1}_1} f\beta $ is a splitting of $\alpha$. Since $F$ is connected, there exists an element $f_1 \in F_1$ and an element $s_1 \in S$ such that $f_1s_1=f_0$. Since $f_1 \in F_1$, $\alpha_1$ refines $(f^{-1}_1\beta)$. Thus
 $$\alpha \vee \bigvee_{f \in F^{-1}} f\beta = \alpha_1 \vee f_0^{-1}\beta= \alpha_1 \vee s_1^{-1}(f_1^{-1}\beta)$$
 is a splitting of $\alpha$. 
\end{proof}



\begin{prop}\label{prop:ball splittings}
Let $\alpha,\beta$ be two combinatorially equivalent generating partitions. Then there is an $n \ge 0$ such that
$$\alpha^n = \bigvee_{g \in B(e,n)} \, g\alpha$$
is a splitting of $\beta$. Here $B(e,n)$ is the ball of radius $n$ centered at the identity element in $G$ with respect to the word metric induced by $S$. Of course, $\alpha^n$ is also a splitting of $\alpha$.
\end{prop}

This proposition is a variation of a result that is well-known in the case $G=\Z$ in the context of subshifts of finite-type. For example, see [LM95, theorem 7.1.2, page 218]. It was first proven in [Wi73].

\begin{proof}
Let $L,M \subset G$ be finite sets such that $\alpha \le \beta^L$ and $\beta \le \alpha^M$. Let $l, m \in \N$ be such that $L \subset B(e,l)$ and $M \subset B(e,m)$. So $\alpha \le \beta^l$ and $\beta \le \alpha^m$. Since balls are connected and $\alpha \le \beta^l$, the previous lemma implies $\beta^l \vee \alpha^{m+l}$ is a splitting of $\beta^l$, and therefore, is a splitting of $\beta$. But $\beta^l \vee \alpha^{m+l} = (\beta \vee \alpha^m)^l = \alpha^{m+l}$.
\end{proof}

\begin{thm}\label{thm:invariant}
Let $\Phi:\sP \to \R$ be any continuous function. Suppose that $\Phi$ is monotone decreasing under splittings; i.e., if $\sigma$ is a splitting of $\alpha$ then $\Phi(\sigma) \le \Phi(\alpha)$. Define $\phi:\sP \to \R$ by 
$$\phi(\alpha)=\lim_{n \to \infty} \Phi(\alpha^n)=\inf_n \Phi(\alpha^n).$$

Then, for any two finite generating partitions $\alpha_1$ and $\alpha_2$, $\phi(\alpha_1)=\phi(\alpha_2)$. So we may define $\phi(G,X,\mu)=\phi(\alpha)$ for any finite generating partition $\alpha$. The number $\phi(G,X,\mu)$ is a measure-conjugacy invariant.
\end{thm}

\begin{proof}
Let $\alpha$ and $\beta$ be two combinatorially equivalent finite partitions. We claim that $\phi(\alpha)=\phi(\beta)$. To see this, suppose, for a contradiction, that $\phi(\alpha)<\phi(\beta)$. Then there exists an $n\ge 0$ such that $\Phi(\alpha^n) <\phi(\beta)$. But by the previous proposition, there is an $m\ge 0$ such that $\beta^m$ is a splitting of $\alpha^n$ which implies $\Phi(\alpha^n) \ge \Phi(\beta^m)\ge \phi(\beta)$, a contradiction. So $\phi(\alpha)=\phi(\beta)$.

For $n\ge 0$ and $\alpha \in \sP$, let $\Phi_n(\alpha)=\Phi(\alpha^n)$. Since $\Phi$ is continuous and the map $\alpha \mapsto \alpha^n$ is also continuous, it follows that $\Phi_n$ is continuous. Since $\phi(\alpha) = \inf_n \Phi_n(\alpha)$, it follows that $\phi$ is upper semi-continuous, i.e., if $\{\beta_n\}$ is a sequence of partitions converging to $\alpha$ then $\limsup_n \phi(\beta_n) \le \phi(\alpha)$. 

Now let $\alpha, \beta$ be two finite generating partitions. By theorem \ref{thm:dense}, there exist finite partitions $\{\beta_n\}_{n=1}^\infty$ combinatorially equivalent to $\beta$ such that $\{\beta_n\}_{n=1}^\infty$ converges to $\alpha$. So $\phi(\beta) = \limsup_n \phi(\beta_n) \le \phi(\alpha)$. Similarly, $\phi(\alpha)\le \phi(\beta)$. So $\phi(\alpha)=\phi(\beta)$.

\end{proof}


\section{The $f$-invariant}\label{sec:f}

In this section, $G=\langle s_1,\ldots,s_r \rangle$. Let $(X,\mu)$ be a standard probability space on which $G$ acts by measure-preserving transformations and let $S=\{s_1^{\pm 1},\ldots,s_r^{\pm 1}\}$. Note $|S|=2r$. Let $F:\sP \to \R$ be defined as in the introduction.



 
\begin{prop}\label{prop:monotone} 
Let $\alpha \in \sP$. If $\sigma$ is a splitting of $\alpha$ then $F(\sigma)\le F(\alpha)$.
\end{prop}
\begin{proof}
By induction, it suffices to prove the proposition in the special case in which $\sigma$ is a simple splitting of $\alpha$. So let $\sigma=\alpha \vee t\beta$ for some $t\in S$ and coarsening $\beta$ of $\alpha$. For any $s \in S$, 
\begin{eqnarray*}
H(\sigma \vee s\sigma)&=&H(\alpha \vee s\alpha) + H(\sigma \vee s\sigma|\alpha \vee s\alpha)\\
&=& H(\alpha \vee s\alpha) + H(s\sigma|\alpha \vee s\alpha) +H(\sigma|\alpha \vee s\alpha \vee s\sigma)\\
&\le &H(\alpha \vee s\alpha) + H(\sigma|\alpha \vee s^{-1}\alpha) +H(\sigma|\alpha \vee s\alpha).
\end{eqnarray*}
The last inequality occurs because $\mu$ is $G$-invariant, so $H(s\sigma|\alpha \vee s\alpha)=H(\sigma|\alpha \vee s^{-1}\alpha)$.

Since $H(\sigma)=H(\alpha) + H(\sigma|\alpha)$, the above implies
\begin{eqnarray*}
F(\sigma) &\le& (1-2r)\big(H(\alpha) + H(\sigma|\alpha)\big) + \sum_{i=1}^r H(\alpha \vee s\alpha) + H(\sigma|\alpha \vee s^{-1}\alpha) +H(\sigma|\alpha \vee s\alpha)\\
&=& F(\alpha) + (1-2r)H(\sigma|\alpha) + \sum_{s\in S} H(\sigma|\alpha \vee s\alpha).
\end{eqnarray*}
Since $\sigma \le \alpha \vee t\alpha$, $H(\sigma | \alpha \vee t\alpha)=0$. Hence
\begin{eqnarray*}
F(\sigma) -F(\alpha) &\le& (1-2r)H(\sigma|\alpha) + \sum_{s\in S-\{t\}} H(\sigma|\alpha \vee s\alpha)\\
&=& \sum_{s\in S-\{t\}}\Big( H(\sigma|\alpha \vee s\alpha) - H(\sigma|\alpha) \Big) \le 0.
\end{eqnarray*}

\end{proof}


Theorem \ref{thm:main} now follows from the proposition above and theorem \ref{thm:invariant}.

\begin{defn}
Let $K$ be a finite set and $\kappa$ a probability measure on $K$. Let $K^G$ be the product space with the product measure $\kappa^G$. The system $(G,K^G,\kappa^G)$ is called the {\bf Bernoulli shift} over $G$ with base measure $\kappa$.

Let $A_k= \{x \in K^G\, |\, x(e)=k\}$ where $e$ denotes the identity element in $G$. Then $\alpha=\{A_k~|~k\in K\}$ is the {\bf Bernoulli partition} associated to $K$. It is generating and $H(\kappa)=H(\alpha)$, by definition.
\end{defn}

\begin{thm}\label{thm:bernoulli}
Let $G=\langle s_1,\ldots,s_r \rangle$ be the free group of rank $r$. Let $K$ be a finite set and $\kappa$ a probability measure on $K$. Then
$$f(G,K^G,\kappa^G) = H(\kappa).$$
\end{thm}

\begin{proof}
Let $\alpha$ be the Bernoulli partition associated to $K$. Let $g_1,\ldots,g_n$ be $n$ distinct elements of $G$. It follows from the Bernoulli condition that the collection $\{g_i\alpha\}_{i=1}^n$ of partitions is independent. This means that if $j:\{1,\ldots,n\} \to K$ is any function then
$$\kappa^G\Big( \bigcap_{i=1}^n ~g_iA_{j(i)} \Big) = \prod_{i=1}^n ~\kappa^G(A_{j(i)}).$$
It is well-known that this implies 
$$H\Big( \bigvee_{i=1}^n g_i\alpha \Big) = \sum_{i=1}^n H(g_i\alpha) = nH(\alpha).$$
See, for example, [Gl03, prop. 14.19, page 257]. So for any $k\ge 1$,
\begin{eqnarray*}
F(\alpha^k) &=& \Big(\frac{1}{2}\sum_{s \in S} H(\alpha^k \vee s\alpha^k)\Big) - (|S|-1)H(\alpha^k)\\
&=& \Big(\frac{1}{2} \sum_{s \in S} |B(e,k) \cup B(s,k)|H(\alpha)\Big) - (|S|-1)|B(e,k)|H(\alpha).
\end{eqnarray*}
Suppose $r >1$. Then, since $G=\langle s_1,\ldots,s_r \rangle$ is free, it is a short exercise to compute:
$$|B(e,k)| = 1 + |S|\frac{(|S|-1)^k -1}{|S|-2},$$
$$|B(e,k)\cup B(s,k)| = 2\frac{(|S|-1)^{k+1} -1}{|S|-2}$$
for all $s\in S$. Thus,
\begin{eqnarray*}
F(\alpha^k) &=& H(\alpha)\Big(|S|\frac{(|S|-1)^{k+1} -1}{|S|-2} - (|S|-1) - (|S|-1) |S|\frac{(|S|-1)^k -1}{|S|-2}\Big)\\
& =& H(\alpha).
\end{eqnarray*}
If $r=1$ then $|B(e,k)|=2k+1$ and $|B(e,k)\cup B(s,k)| = 2k+2$. So it follows in a similar way that $F(\alpha^k)=H(\alpha)$. So $f(G,X,\mu)=\lim_{k \to\infty} F(\alpha^k)=H(\alpha)=H(\kappa)$. 
\end{proof}

\begin{proof}[Proof of theorem \ref{thm:free}]
According to Stepin's theorem [St75], if $(K_1,\kappa_1), (K_2,\kappa_2)$ are standard Borel probability spaces with $H(\kappa_1)=H(\kappa_2)$ then $(G,K_1^G,\kappa_1^G)$ is measurably conjugate to $(G,K_2^G,\kappa_2^G)$. 

Now suppose $(K_1,\kappa_1)$, $(K_2,\kappa_2)$ are Borel probability spaces such that $(G,K_1^G,\kappa_1^G)$ is measurably conjugate to $(G,K_2^G,\kappa_2^G)$. Let $(L_1,\lambda_1), (L_2,\lambda_2)$ be probability spaces with $|L_1| + |L_2|<\infty$ and $H(\lambda_i) = H(\kappa_i)$ for $i=1,2$. By Stepin's theorem, $(G,L_i^G,\lambda_i^G)$ is measurably conjugate to $(G,K_i^G,\kappa_i^G)$. By the above theorem, $f(G,L_i^G,\lambda_i^G)=H(\lambda_i)$. Since $f$ is a measure-conjugacy invariant, $H(\kappa_1)=H(\kappa_2)$.
\end{proof}

\end{document}